\RequirePackage{cmap}
\documentclass[11pt]{amsart}
\usepackage{amssymb,amsmath,amsfonts,amscd,euscript}

\newcommand{\nc}{\newcommand}

\numberwithin{equation}{section}
\newtheorem{thm}{Theorem}[section]
\newtheorem*{thm*}{Theorem}

\newtheorem{lem}[thm]{Lemma}
\newtheorem{exam}[thm]{Example}
\newtheorem{cor}[thm]{Corollary}
\newtheorem*{cor*}{Corollary}
\theoremstyle{remark}
\newtheorem{rem}[thm]{Remark}
\newtheorem{definition}[thm]{Definition}

\newcount\cols {\catcode`,=\active\catcode`|=\active
 \gdef\Young(#1){\hbox{$\vcenter
 {\mathcode`,="8000\mathcode`|="8000
  \def,{\global\advance\cols by 1 &}%
  \def|{\cr
        \multispan{\the\cols}\hrulefill\cr
        &\global\cols=2 }%
  \offinterlineskip\everycr{}\tabskip=0pt
  \dimen0=\ht\strutbox \advance\dimen0 by \dp\strutbox
  \halign
   {\vrule height \ht\strutbox depth \dp\strutbox##
    &&\hbox to \dimen0{\hss$##$\hss}\vrule\cr
    \noalign{\hrule}&\global\cols=2 #1\crcr
    \multispan{\the\cols}\hrulefill\cr%
   }
 }$}}
 \gdef\Skew(#1:#2){\hbox{$\vcenter
 {\mathcode`,="8000\mathcode`|="8000
  \dimen0=\ht\strutbox \advance\dimen0 by \dp\strutbox
  \def\boxbeg{\vbox
    \bgroup\hrule\kern-0.4pt\hbox to\dimen0\bgroup\strut\vrule\hss$}%
  \def\boxend{$\hss\egroup\hrule\egroup}%
  \def,{\boxend\boxbeg}%
  \def|##1:{\boxend\vrule\egroup\nointerlineskip\kern-0.4pt
    \moveright##1\dimen0\hbox\bgroup\boxbeg}%
  \def\\##1\\##2:{\boxend\vrule\egroup\nointerlineskip\kern-0.4pt
    \kern ##1\dimen0\moveright##2\dimen0\hbox\bgroup\boxbeg}%
  \moveright#1\dimen0\hbox\bgroup\boxbeg#2\boxend\vrule\egroup
 }$}} }

\def\smallsquares
 {\textfont0=\scriptfont0 \scriptfont0=\scriptscriptfont0
  \textfont1=\scriptfont1 \scriptfont1=\scriptscriptfont1
  \setbox0=\hbox{$($}
  \setbox\strutbox=\hbox{\vrule width 0pt height\ht0 depth\dp0 }
 }

\nc{\gl}{\mathfrak{gl}}
\nc{\GL}{\mathfrak{GL}}
\nc{\g}{\mathfrak{g}}
\nc{\gh}{\widehat\g}
\nc{\h}{\mathfrak{h}}
\nc{\n}{\mathfrak{n}}
\nc{\la}{\lambda}
\nc{\C}{\mathbb C }
\nc{\D}{\mathbb D }
\nc{\Z}{\mathbb Z }
\nc{\N}{\mathbb N }
\nc{\R}{\mathbb R }
\nc{\Q}{\mathbb Q }
\nc{\al}{\alpha }
\nc{\bs}{{\bf s}}
\nc{\bt}{{\bf t}}
\nc{\br}{{\bf r}}

\nc{\ch}{{\mathop {\rm ch}}}
\nc{\Tr}{{\mathop {\rm Tr}\,}}
\nc{\Id}{{\mathop {\rm Id}}}
\nc{\ad}{{\mathop {\rm ad}}}
\nc{\bra}{\langle}
\nc{\ket}{\rangle}
\nc{\bx}{{\bf x}}
\nc{\pa}{\partial}
\nc{\ld}{\ldots}
\nc{\cd}{\cdots}
\nc{\gr}{\mathrm{gr}}
\nc{\ov}{\overline}

\nc{\msl}{\mathfrak{sl}}
\nc{\msp}{\mathfrak{sp}}
\nc{\mgl}{\mathfrak{gl}}
\nc{\mso}{\mathfrak{so}}

\newcommand{\ol}{\overline}
\newcommand{\bc}{{\mathbb C}}
\newcommand{\bz}{{\mathbb Z}}

\def\gr{\operatorname{gr}}
\def\wt{\operatorname{wt}}

\usepackage[style=numeric, sorting=none, maxnames=20]{biblatex}
\begin{document}
\title{FFLV polytopes for odd symplectic Lie algebras}

\author{Dmitry Rybin}
\address{Dmitry Rybin:\newline
International Laboratory for Mirror Symmetry and Automorphic Forms,\newline
National Research University Higher School of Economics, Moscow
}
\email{darybin@edu.hse.ru}

\begin{abstract}
We consider ``odd symplectic Lie algebras'' defined in terms of maximal rank skew-symmetric forms. We provide FFLV polytopes for these algebras and prove their standard properties. In particular, we obtain a new graded character formula and PBW bases for an analogue of irreducible representations in terms of lattice points of FFLV polytopes.
\end{abstract}
\maketitle

\section*{Introduction}
Feigin--Fourier--Littelmann--Vinberg (FFLV) polytopes are certain convex polytopes which appear in the study of representations of Lie algebras. There are several motivations to research FFLV polytopes. First, they provide information about the PBW filtration and PBW bases, which have important applications in algebraic geometry, representation theory and mathematical physics. Second, they give combinatorial formulas for graded characters of representations of Lie algebras.

The study of FFLV polytopes initiated with papers \cite{FFL1}, \cite{FFL2} investigating PBW filtration and bases for Lie algebras $\msl_{n}$ and $\msp_{2n}$, and proving Vinberg conjecture \cite{Vinberg}. The case of $\mso_{2n + 1}$ was handled in \cite{IM} with certain modifications in the definition of polytopes and bases. Recently a preprint providing FFLV type bases for Demazure modules has been posted \cite{Kambaso}.

In \cite{Proctor} Proctor proposed a definition of odd symplectic groups $Sp_{2n + 1}$ and defined an analogue of highest weight irreducible representations $\mathcal{V}_{0}(\la)$ for these groups. One of the important results of his work is a Gelfand--Zetlin type character formula for $\mathcal{V}_{0}(\la)$. However, his main motivation was the product formula for dimensions of $\mathcal{V}_{0}(\la)$, which is also equal to a multivariate Ehrhart polynomial of a family of order poset polytopes. Product formulas for Ehrhart polynomials of ordet poset polytopes are still not fully understood. See a survey \cite{PosetDynamics} on a single variable product formulas.

According to the work \cite{ABS} there is a bijection between lattice points of marked order polytopes (Gelfand--Zetlin type polytopes) and lattice points of the corresponding marked chain polytopes (FFLV type polytopes). Thus we can define FFLV polytopes for $\msp_{2n + 1}$ based on the work of Proctor by passing to marked chain polytopes. It is natural to ask whether these polytopes share the same properties as in $\msl_{n}$ and $\msp_{2n}$ cases. More precisely, first question is whether lattice points of marked chain polytopes of $\msp_{2n + 1}$ explicitly enumerate a PBW basis of $\gr \mathcal{V}_{0}(\la)$ and $\mathcal{V}_{0}(\la)$. Second question is whether these polytopes have the Minkowski sum property, that was proven for FFLV polytopes of $\msl_n$ and $\msp_{2n}$ in papers \cite{FFL1} and \cite{FFL2}, and further studied in \cite{Fang}, \cite{MinkowskiSums}. In this paper we affirmatively answer these questions.

Our contribution is the translation of all results obtained for $\msp_{2n}$ in \cite{FFL2} to the $\msp_{2n + 1}$ case. In particular, we introduce new marked chain polytopes coming from representation theory. We generalize the argument from \cite{FFL2} that allows us to provide PBW bases for indecomposable trace-free representations of $\msp_{2n + 1}$ and irreducible representations of $\msp_{2n}$ simultaneously.

\section*{Acknowledgements}
The author is partially supported by Laboratory of Mirror Symmetry NRU HSE, RF Government grant, ag. № 14.641.31.0001. The author thanks Evgeny Feigin for discussions and directions.

\section{Definitions}
\subsection{Lie algebras}
Let $U(\g)$ and $S(\g)$ denote the universal enveloping and symmetric algebras of a Lie algebra $\g$ respectively. One of the objects of interest is the PBW filtration. For any Lie algebra $\g$ consider the filtration on $\mathrm{U}(\g)$ with degree $s$ subspace defined as:
$$U^{s}(\g) = \mathrm{span}\left\lbrace \prod_{i = 1}^{k} x_{i} \,|\, x_{i} \in \g, 0 \leq k \leq s \right\rbrace.$$
Now assume that $\g$ has a Cartan decomposition $\n^{+} \oplus \h \oplus \n^{-}$. Assume that $V(\la)$ is a highest weight representation of $\g$ with highest weight vector $v_{\la}$. Using the action of $U(\g)$ on $V(\la)$, we can define the PBW filtration on $V(\la)$.
\begin{definition}
PBW filtration on the highest weight representation $V(\la)$ is defined as
$$V^{s}(\la) = U^{s}(\g) v_{\la} = U^{s}(\n^{-}) v_{\la}.$$
\end{definition}
The filtration on the module depends on the choice of positive roots. However, it can be shown that the following quantities do not depend on this choice.
\begin{definition} The $q$-dimension of the highest weight representation \hspace{-0.6mm}$V(\la)$\hspace{-0.59mm} is
$$\dim_q V(\la)  = \sum_{s} \dim \thinspace V^{s}(\la) / V^{s - 1}(\la) \cdot q^{s}.$$
\end{definition}
\begin{definition}
The $q$-character of the highest weight representation $V(\la)$ is $$\ch_q V(\la)  = \sum\limits_{\mu, s}  \dim V^{s}(\la)^{\mu}/ V^{s - 1}(\la)^{\mu} \cdot e^{\mu},$$
where the summation is over all weight subspaces $V(\la)^{\mu}$.
\end{definition}
\noindent Note that the associated graded space $\gr V(\la)$ can be considered as an $S(\g)$-module.

Following the notation from \cite{FFL2} we set $R^{+}$ to be the set of positive roots of $\msp_{2n}$. Let $\al_i$ and $\omega_i$ be simple roots and the fundamental weights of $\msp_{2n}$. The elements of $R^{+}$ have the form
\begin{gather*}
\al_{i,j} = \al_i + \al_{i + 1} + \dots + \al_j,\ 1 \le i \le j \le n - 1,\\
\alpha_{i, \ol{j}} = \alpha_i + \alpha_{i + 1} + \ldots +
\alpha_n + \alpha_{n - 1} + \ldots + \alpha_j, \ 1 \le i \le j \le n.
\end{gather*}
Unlike the paper \cite{FFL2} we do not identify $\al_{i, n}$ and $\al_{i, \ol{n}}$. We reserve this notation on purpose for other roots. Let $\msp_{2n} = \n^{+} \oplus \h \oplus \n^{-}$ be a Cartan decomposition. Denote by $f_{\al} \in \n^{-}_{-\al}$ the root vector of $\msp_{2n}$.

\subsection{Odd symplectic Lie algebra}
Let $V = \bc^{2n + 1}$ and let $e_1, ..., e_{2n + 1}$ be a fixed basis. Let $x_1, ..., x_{2n + 1}$ be coordinate functions. Denote by $\omega$ a maximal rank skew-symmetric form on $V$ given by $\omega = \sum_{i=1}^{n} dx_{2i - 1} \wedge dx_{2i}$. Following the work \cite{Proctor} we give a definition.
\begin{definition}
Odd symplectic group $Sp_{2n + 1}$ is a subgroup of $GL(V)$ that preserves a skew symmetric bilinear form on $V$ of maximal rank. Odd symplectic Lie algebra $\msp_{2n + 1}$ is a Lie algebra of the group $Sp_{2n + 1}$.
\end{definition}
We can assume that $\omega$ is the preserved form, then 
$$\msp_{2n + 1} \cong \bc^{2n} \rtimes (\msp_{2n} \oplus \gl_1)$$
and in a fixed basis $\msp_{2n + 1}$ has a block matrix form,
$$
\msp_{2n + 1} = \begin{bmatrix}
\msp_{2n} & 0 \\
* & z
\end{bmatrix}.
$$
Note that $\msp_{2n + 1} \not\subset \msp_{2n + 2}$. Let $E_{i, j}$ denote the usual matrix unit. Then simple roots of the subalgebra $\msp_{2n}$ can be chosen as $h_{\al_i} = E_{i, i} - E_{n + i, n + i}$. Apart from roots of the subalgebra $\msp_{2n}$ the Lie algebra $\msp_{2n + 1}$ has roots defined by elements $E_{2n + 1, n + i}$ and $E_{2n + 1, i}, \; i = 1, ..., n$. Let us denote the root of $\msp_{2n + 1}$ defined by the element $E_{2n + 1, 2n}$ as $\tilde\al_{n + 1}$. Using the identities
$$[E_{i, i} - E_{n + i, n + i}, E_{2n + 1, i}] = -E_{2n + 1, i},$$
$$[E_{i, i} - E_{n + i, n + i}, E_{2n + 1, i}] = E_{2n + 1, n + i},$$
we conclude that $E_{2n + 1, n + i}$ corresponds to the root $\al_i + ... + \al_{n - 1} - \tilde\al_{n + 1}$ and $E_{2n + 1, i}$ corresponds to the root $-\al_i - ... - \al_n - \tilde\al_{n + 1}$. We introduce the notations
$$\al_{i, n} = \al_i + ... + \al_n - \tilde\al_{n + 1},$$
$$f_{i, n} = f_{\al_{i, n}} = E_{2n + 1, i}.$$
We also define subspaces
$$F^{+} = \bigoplus\limits_{i = 1}^{n} \bc \cdot E_{2n + 1, n + i},\,\, F^{-} = \bigoplus\limits_{i = 1}^{n} \bc \cdot E_{2n + 1, i}.$$
Consider nilpotent subalgebras $\tilde\n^{\pm} = \n^{\pm} \oplus F^{\pm}$ and Cartan subalgebra 
$$\tilde\h = \h \oplus \bc \cdot E_{2n + 1, 2n + 1}.$$
We will use decomposition $\msp_{2n + 1}  = \tilde\n^{-} \oplus \tilde\h \oplus \tilde\n^{+}$ as the analogue of Cartan decomposition. To construct the analogue of irreducible representations for $\msp_{2n + 1}$ Proctor used trace-free tensors construction.
\begin{definition}
A k-tensor $v \in V^{\otimes k}$ is called trace-free with respect to $\omega$ if all of $k(k - 1)/2$ contractions of $v$ with the form $\omega \in V^{* \otimes 2}$ are equal to 0.
\end{definition}
Let $\la$ be a Young diagram with $n$ non-zero rows. Let $V^{\la} \subset V^{\otimes |\lambda|}$ be a subrepresentation of $GL(V)$ of highest weight $\la$. Denote by $\mathcal{V}_{0}^{\la}$ the subspace of trace-free tensors of $V^{\la}$. The highest weight vector $v_{\la} \in V^{\la}$ belongs to the subspace of trace-free tensors $\mathcal{V}_{0}(\la)$. It was proven in \cite{Proctor} that $\mathcal{V}_{0}(\la)$ is an indecomposable representation of $\msp_{2n + 1}$ with highest weight $\la$. The analogous description of the irreducible representations of classical groups as subspaces of trace-free tensors can be found in \cite{Weyl}.

For Young diagrams $\mu, \lambda$ denote by $\mu \prec \lambda$ the interlacing conditions: $\la_{i + 1} \leq \mu_i \leq \la_i \;\; \forall i$. Denote by $V_{\msp_{2n}}(\la)$ the irreducible representation of $\msp_{2n}$ of highest weight $\la$. The following important connection between representations of $\msp_{2n}$ and $\msp_{2n + 1}$ was given in \cite{Proctor}. 
\begin{lem}\label{BranchingRule}
The representation $\mathcal{V}_{0}(\la)$ under the restriction to subalgebra $\msp_{2n} \oplus \mgl_1$ is decomposed as 
$$\mathcal{V}_{0}(\la) = \bigoplus_{\mu \prec \la} V_{\msp_{2n}}(\mu) \otimes \mathbb{C}_{|\la| - |\mu|},$$
where $\mathbb{C}_{|\la| - |\mu|}$ is the one dimensional representation with $\mgl_1$ weight $|\la| - |\mu|$.
\end{lem}

Proctor considered a generalization $\msp_{2n, k} = (\msp_{2n} \oplus \gl_k) \ltimes \bc^{2nk}$ with $k > 1$ given by selection of rank $2n$ skew-symmetric form in $\bc^{2n + k}$. However, the product formula and representation theory backing it exists only for four families of examples:
\begin{enumerate}
    \item Case $k = 0$ gives $\msp_{2n}$ with FFLV polytopes obtained in \cite{FFL2},
    \item Case $k = 1$ gives $\msp_{2n + 1}$ considered in this paper,
    \item Case $n = 0$ gives $\gl_k$ with FFLV polytopes obtained in \cite{FFL1},
    \item Case $n = 1$ gives $(\msp_1 \oplus \gl_k) \ltimes \bc^{2m}$.
\end{enumerate}
Poset for the case $n = 1$ is obtained from $A_k$ poset by adding one extra element. The corresponding product formula for multivariate Ehrhart polynomial of marked order (or chain) polytope for weight $\la = \sum_{i=1}^{k - 1} m_i \omega_i$ contains an additional linear factor:
$$\frac{k + \sum_{i=1}^{k - 1} i \cdot m_i}{k} \cdot \prod\limits_{i \leq j} \frac{m_i + m_{i + 1} + ... + m_j + j - i + 1}{j - i + 1}.$$

\subsection{FFLV polytopes}
Recall the definition of FFLV polytope for $\msp_{2n}$. Consider the poset formed by all positive roots $R^{+}$ of $\msp_{2n}$ (case $n = 3$ on the diagram):
$$
\begin{array}[pos]{ccccccccccccc}
&& \alpha_{1,1} & \rightarrow & \alpha_{1,2}& \rightarrow & \alpha_{1,\ol{3}}& \rightarrow  & \alpha_{1,\ol{2}}& \rightarrow & \alpha_{1,\ol{1}} &&\\
&& && \downarrow && \downarrow && \downarrow && \\
&& &  & \alpha_{2,2}& \rightarrow & \alpha_{2,\ol{3}}& \rightarrow & \alpha_{2,\ol{2}}&  &  &&\\
&& &&  && \downarrow && && && \\
&& &  &  &  & \alpha_{3,\ol{3}} &  &  &  &&
\end{array}
$$
The polytope $FFLV_{2n}(\la)$ is defined in space $\mathbb{R}_{\geq 0}^{n^2}$ with coordinates $(s_{\al})_{\al \in R^{+}}$ by explicit set of inequalities. There is one inequality for each symplectic Dyck path. Symplectic Dyck path is any maximal chain $p = (p(0), ..., p(k))$ in the root poset that starts at the root $p(0) = \al_{i, i}$ for some $i$ and ends either at $\al_{j, j}$ or at $\al_{j, \ol{j}}$ for some $j \geq i$. Let $\D_{2n}$ denote the set of all symplectic Dyck paths for $\msp_{2n}$ root poset.

\begin{equation*}
FFLV_{2n}(\la):= \bigg\{ \bs \mid \forall p \in \D_{2n}:
\begin{array}{l}
\text{ If } p(0) = \al_{i, i}, p(k) = \al_{j, j}, \text{ then }\\
\sum_{\al \in p} s_{\al} \leq m_i + \dots + m_j,\\
\text{ if } p(0) = \al_{i, i}, p(k) = \al_{j, \ol{j}}, \text{ then }\\
\sum_{\al \in p} s_{\al} \leq m_i + \dots + m_n\\
\end{array}\bigg\}.
\end{equation*}
To shorten the notation we use $s_{i,j}$ for $s_{\al_{i, j}}$, and $s_{i,\ol{j}}$ for $s_{\al_{i, \ol{j}}}$. 
\begin{exam}
The polytope $FFLV_{4}(m_1\omega_1 + m_2\omega_2)$ has four coordinates $s_{1, 1}, s_{1, \ol{2}}, s_{2, \ol{2}}, s_{1, \ol{1}} \in \mathbb{R}_{\geq 0}$ constrained by inequalities
$$s_{1, 1} \leq m_1, \;\;\; s_{2, \ol{2}} \leq m_2,$$
$$s_{1, 1} + s_{1, \ol{2}} + s_{1, \ol{1}} \leq m_1 + m_2, \;\;\; s_{1, 1} + s_{1, \ol{2}} + s_{2, \ol{2}} \leq m_1 + m_2.$$
\end{exam}

For any lattice point $\bs = (s_{\al})_{\al \in R^{+}} \in FFLV_{2n}(\la)$ denote by $f^{\bs}$ the element of the universal enveloping algebra
$$f^{\bs} = \prod\limits_{\al} f_{\al}^{s_{\al}} \in U(\n^{-}) \subset U(\msp_{2n}).$$
The product depends on the ordering of the roots $\al$. We suppose that arbitrary fixed ordering is used. Note that all orderings give equivalent elements when we pass from $U(\n^{-})$ to the associated graded $S(\n^{-})$. 

We denote the set of integral points of a polytope $P$ as $P^{\bz} \subset P$. The following result was proven in \cite{FFL2}.
\begin{thm}\label{EvenFFLV}
For dominant integral weight $\la$ the set
$$f^{\bs} v_{\la}, \bs \in FFLV^{\bz}_{2n}(\la),$$
is a basis of the irreducible representation $V_{\msp_{2n}}(\la)$ and of the associated graded space $\gr V_{\msp_{2n}}(\la)$ with respect to PBW filtration.
\end{thm}
\noindent As a corollary, this result provides a combinatorial formula for the graded character. Introduce the notation 
$$\wt(\bs) = \sum_{1 \leq i \leq j < n} \al_{i, j}s_{i, j} + \sum_{1 \leq i \leq j \leq n} \al_{i, \ol{j}}s_{i, \ol{j}}, \;\; \deg(\bs) = \sum_{1 \leq i \leq j < n} s_{i, j} + \sum_{1 \leq i \leq j \leq n} s_{i, \ol{j}}.$$
Then
$$\ch_{q} V_{\msp_{2n}}(\la) = \sum_{\bs \in FFLV^{\bz}_{2n}(\la)} e^{\la - \wt(\bs)} q^{\deg(\bs)}.$$
The following results were also obtained in \cite{FFL2}.
\begin{thm}\label{EvenSubmodule}
The $S(\n^{-})$-module generated by cyclic vector $$v_{\la} \otimes v_{\mu} \in \gr V_{\msp_{2n}}(\la) \otimes \gr V_{\msp_{2n}}(\mu)$$ is isomorphic to $\gr V_{\msp_{2n}}(\la + \mu)$.
\end{thm}
\begin{thm}\label{EvenMinkowski}
We have an equality of the Minkowski sums of lattice points
$$FFLV_{2n}^{\bz}(\la) + FFLV_{2n}^{\bz}(\mu) = FFLV_{2n}^{\bz}(\la + \mu).$$
\end{thm}

We introduce the root poset for $\msp_{2n + 1}$ that we will work with (case $n = 3$ on the diagram):
$$
\begin{array}[pos]{ccccccccccccccc}
&& \alpha_{1,1} & \rightarrow & \alpha_{1,2} & \rightarrow & \alpha_{1,3} & \rightarrow & \alpha_{1,\ol{3}}& \rightarrow  & \alpha_{1,\ol{2}}& \rightarrow & \alpha_{1,\ol{1}} &&\\
&& && \downarrow && \downarrow && \downarrow && \downarrow && \\
&& &  & \alpha_{2,2}& \rightarrow & \alpha_{2,3}&\rightarrow & \alpha_{2,\ol{3}}& \rightarrow & \alpha_{2,\ol{2}}&  &  &&\\
&& &&  && \downarrow && \downarrow && && && \\
&& &  &  &  & \alpha_{3,3} & \rightarrow & \alpha_{3,\ol{3}} &  &  &  &&
\end{array}
$$
The polytope $FFLV_{2n + 1}(\la)$ is defined in space $\R^{n(n + 1)}_{\geq 0}$ with coordinates $s_\al$. The defining inequalities are completely analogous to the $\msp_{2n}$ case. Let $\D_{2n + 1}$ be the set of symplectic Dyck paths for $\msp_{2n + 1}$ root poset. That is, each element $p \in \D_{2n + 1}$ is a maximal chain in $\msp_{2n + 1}$ root poset that starts at $p(0) = \al_{i , i}$ for some $i$ and ends either at $\al_{j, j}$ or at $\al_{j, \ol{j}}$ for some $j \geq i$.
\begin{equation*}
FFLV_{2n + 1}(\la):= \bigg\{ \bs \mid \forall p \in \D_{2n + 1}:
\begin{array}{l}
\text{ If } p(0) = \al_{i, i}, p(k) = \al_{j, j}, \text{ then }\\
\sum_{\al \in p} s_{\al} \leq m_i + \dots + m_j,\\
\text{ if } p(0) = \al_{i, i}, p(k) = \al_{j, \ol{j}}, \text{ then }\\
\sum_{\al \in p} s_{\al} \leq m_i + \dots + m_n\\
\end{array}\bigg\}.
\end{equation*}
\begin{exam}
The polytope $FFLV_{5}(m_1\omega_1 + m_2\omega_2)$ has six coordinates $s_{1, 1}, s_{1, 2}, s_{1, \ol{2}}, s_{1, \ol{1}}, s_{2, 2}, s_{2, \ol{2}} \in \mathbb{R}_{\geq 0}$ constrained by inequalities
$$s_{1, 1} \leq m_1, \;\;\; s_{2, 2} \leq m_2, \;\;\; s_{2, 2} + s_{2, \ol{2}} \leq m_2,$$
$$s_{1, 1} + s_{1, 2} + s_{2, 2} \leq m_1 + m_2, \;\;\; s_{1, 1} + s_{1, 2} + s_{1, \ol{2}} + s_{1, \ol{1}} \leq m_1 + m_2,$$
$$s_{1, 1} + s_{1, 2} + s_{1, \ol{2}} + s_{2, \ol{2}} \leq m_1 + m_2, \;\;\; s_{1, 1} + s_{1, 2} + s_{2, 2} + s_{2, \ol{2}} \leq m_1 + m_2.$$
\end{exam}
We will prove the following results.
\begin{thm}\label{OddFFLV}
Let $\la = m_1 \omega_1 + ... + m_n \omega_n$ be an integral weight of $\msp_{2n + 1}$ with $m_i \in \bz_{\geq 0}$ for all $i$. Then the set
$$f^{\bs} v_{\la}, \bs \in FFLV^{\bz}_{2n + 1}(\la),$$
is a basis of the indecomposable trace-free representation $\mathcal{V}_{0}(\la)$ and of the associated graded space $\gr \mathcal{V}_{0}(\la)$ with respect to PBW filtration.
\end{thm}
\begin{thm}\label{OddSubmodule}
The $S(\tilde\n^{-})$-module generated by cyclic vector $$v_{\la} \otimes v_{\mu} \in \gr \mathcal{V}_{0}(\la) \otimes \gr \mathcal{V}_{0}(\mu)$$ is isomorphic to $\gr \mathcal{V}_{0}(\la + \mu)$.
\end{thm}
\begin{thm}\label{OddMinkowski}
We have an equality of the Minkowski sums of sets
$$FFLV_{2n + 1}^{\bz}(\la) + FFLV_{2n + 1}^{\bz}(\mu) = FFLV_{2n + 1}^{\bz}(\la + \mu).$$
\end{thm}
To prove these results we first demonstrate basic lemmas about the structure of representations $\mathcal{V}_{0}(\la)$. In particular, we obtain a simpler basis given by monomials from $U(\tilde\n^{-})$. The proof strategy of Theorem \ref{EvenFFLV} from \cite{FFL2} requires modifications in order to be applied to Theorem \ref{OddFFLV}. We show how to modify the strategy. In fact, our modification is a slight generalization because the Theorem \ref{EvenFFLV} follows from our results.

\begin{rem}
In the work \cite{Proctor} the highest weight $\la$ of $\mathcal{V}_{0}(\la)$ has height up to $n + 1$. However, representations with $\la_{n + 1} \neq 0$ contribute no non-trivial statements since they are obtained from those with $\la_{n + 1} = 0$ by tensoring with $\mgl_1$ representation $\bc_{|\la_{n + 1}|}$.
\end{rem}

\section{Basic structure of $\mathcal{V}_{0}(\la)$}
The branching rule given by Lemma \ref{BranchingRule} implies that $\msp_{2n}$ subrepresentations of $\mathcal{V}_{0}(\la)$ are enumerated by tuples $$(\tilde\mu_1, ..., \tilde\mu_n) \in \bz^{n},\; 0 \leq \tilde\mu_i \leq m_i \;\; \forall i.$$
Let $\delta(\la)$ denote the set of such tuples. We also denote by $f^{\tilde\mu}$ the monomial
$$f^{\tilde\mu} = \prod\limits_{i = 1}^{n} E_{2n + 1, i}^{\tilde\mu_i} \in U(\tilde\n^{-}).$$

As a result of this section we will prove the following claim.
\begin{thm}\label{ExtraThm}
Let $\la = m_1 \omega_1 + ... + m_n \omega_n$ be an integral weight of $\msp_{2n + 1}$ with $m_i \geq 0$ for all $i$. Then the set
$$f^{\bs}f^{\tilde\mu} v_{\la},\; \tilde\mu \in \delta(\la),\; \bs \in FFLV^{\bz}_{2n}(\la - \tilde\mu),$$
is a basis of the indecomposable trace-free representation $\mathcal{V}_{0}(\la)$ and of the associated graded space $\gr \mathcal{V}_{0}(\la)$ with respect to PBW filtration.
\end{thm}
Due to isomorphism of $\msp_{2n}$ representations (Lemma \ref{BranchingRule})
$$\mathcal{V}_{0}(\la) \cong \bigoplus_{\mu \prec \la} V_{\msp_{2n}}(\mu),$$
it is enough to prove the spanning property. The linear independence will follow from dimension considerations.

We separate the proof into several lemmas. First, we prove that spanning property follows from non-vanishing of projections of certain vectors on $V_{\msp_{2n}}(\mu)$. Second, we prove that the vector $v_{\la}$ is cyclic in $\mathcal{V}_0(\la)$. Finally, we show that Theorem \ref{ExtraThm} follows from these statements.

\begin{lem}\label{Lemma1}
Suppose the projection of $f^{\tilde\mu}v_{\la}$ on $V_{\mathfrak{sp}_{2n}}(\la - \tilde\mu)$ is non-zero for all $\tilde\mu \in \delta(\la)$. Then Theorem \ref{ExtraThm} is true.
\end{lem}
\begin{proof}
Let us prove that the set of vectors from Theorem \ref{ExtraThm} spans all subrepresentations $V_{\mathfrak{sp}_{2n}}(\mu)$, $\mu \prec \la$. The proof will go by decreasing induction on $\mu$, where the order $\prec$ on the highest weights is the standard one. 

Induction base: if $\mu$ is maximal with respect to $\prec$, then $\mu = \la$ and $f^{\la - \la}v_{\la} = v_{\la}$. The spanning property follows from the spanning property of the set $f^{\bs} v_{\la}, \; \bs \in FFLV^{\bz}_{2n}(\la)$.

Induction step: by assumption the vector $f^{\lambda - \mu}v_{\lambda}$ is non-zero. Its weight with respect to the subalgebra $\msp_{2n} \oplus \mathfrak{gl}_1$ is equal to $(\mu, |\la| - |\mu|)$. By definition of $\prec$ such weight subspace can only be present in representations $V_{\msp_{2n}}(\nu)$ with $\mu \prec \nu$. By induction we can make a linear combination with $f^{\lambda - \mu}v_{\lambda}$ such that it has weight $\mu$ and lies in $V_{\msp_{2n}}(\mu)$. Here we used the fact that the representation $\mathcal{V}_{0}(\la)$ has no multiplicities when restricted to $\msp_{2n}$. The step is finished by the spanning property of the set $f^{\bs} v_{\la}, \; \bs \in FFLV^{\bz}_{2n}(\mu)$.
\end{proof}

\begin{lem}\label{Lemma2}
The highest weight vector $v_{\la}$ is cyclic in representation $\mathcal{V}_{0}(\la)$ with respect to the action of $\msp_{2n + 1}$.
\end{lem}
\begin{proof} Let $v_{\mu}$ be the highest weight vector of subrepresentation $V_{\msp_{2n}}(\mu)$ of subalgebra $\msp_{2n}$. It is enough to show that $v_{\mu}$ is contained in the orbit of $v_{\la}$. Let $W = \langle e_1, ..., e_n, e_{2n + 1} \rangle$. Note that $v_{\mu} \in W^{\otimes |\la|} \subset V^{\otimes |\la|}$, because its weight with respect to $GL(V)$ action is equal to 
$$(\mu_1, \mu_2, ..., \mu_n, 0, ..., 0, |\la| - |\mu|).$$
Introduce the subspace $$W^{\la} = W^{\otimes |\la|} \cap \mathcal{V}_{0}(\la).$$
Note that the form $\omega$ vanishes on any pair of vectors from $W$. Therefore $W^{\otimes |\la|}$ consists of trace-free elements and we can write
$$W^{\la} = W^{\otimes |\la|} \cap V^{\la}.$$
By definition $V^{\la}$ is equal to the image of $V^{\otimes |\la|}$ under the action of Young symmetrizer $c_{\la}$. We claim that $W^{\lambda} = W^{\otimes |\la|} \cap V^{\la} = c_{\la}(W^{\otimes |\la|})$. The last equality is evident from the consideration of decomposable tensors formed by basis vectors. It follows that $W^{\la} = c_{\la}(W^{\otimes |\la|})$ is isomorphic to the irreducible representation of $\mathfrak{gl}(W)$ of highest weight $\la$. However, there is a subalgebra $\g \subset \mgl(W) \subset \msp_{2n + 1}$ that contains a lower-triangular subalgebra of $\mgl(W)$. It follows that the vector $v_{\la} \in W^{\la} \subset V^{\la}$ is cyclic in $W^{\la}$ with respect to the action of $\mathfrak{g} \subset \mathfrak{gl}(W)$. We conclude that its orbit contains all vectors $v_{\mu}$ and the Lemma follows.
\end{proof}

\noindent We are ready to prove Theorem \ref{ExtraThm}.
\begin{proof}
By Lemma \ref{Lemma2} the vector $v_{\la}$ is cyclic. It vanishes under the action of $\tilde\n^{+}$, thus by commutations we find that $U(\tilde\n^{-}) v_{\la} = \mathcal{V}_{0}(\la)$. Since $$[\n^{-}, F^{-}] \subset F^{-},$$
by commutations we can obtain
$$U(\n^{-})U(F^{-}) v_{\la} = \mathcal{V}_{0}(\la).$$
Therefore $v_{\mu}$ can be written as
$$v_{\mu} = \sum_{j} n_j f_j v_{\la},\; f_j \in U(F^{-}),\; n_j \in U(\n^{-}).$$
However, the vector $v_{\mu}$ does not belong to $\n^{-} \cdot U(\n^{-}) \mathcal{V}_{0}(\la)$ by definition of the highest weight vector. Thus at least for one $j$ we have $n_j = 1$. By weight computation we find that $f_j = C f^{\la - \mu}$ for some constant $C \in \bc$. By noting that $v_{\mu} - f_j v_{\la} \in \n^{-}\cdot U(\n^{-})U(F^{-})\mathcal{V}_0(\la)$ we conclude that the projection of $f^{\la - \mu} v_{\la}$ on $V(\mu)$ is non-zero. We finish the proof by Lemma \ref{Lemma1}.
\end{proof}
\noindent As a corollary we obtain a combinatorial graded character formula.
\begin{cor}\label{SimpleqChar}
$$\ch_{q} \mathcal{V}_{0}(\la) = \sum\limits_{\substack{\tilde\mu \in \delta(\la),\\ \bs \in FFLV^{\bz}_{2n}(\la - \tilde\mu)}} e^{\la -  \wt(\tilde\mu) - \wt(\bs)} q^{\deg(\bs) + \deg(\tilde\mu)}.$$
\end{cor}
\noindent Another corollary is Theorem \ref{OddSubmodule}.
\begin{proof}
We have an equality of $S(\n^{-})$-modules
$$\gr \mathcal{V}_{0}(\la) = \bigoplus_{\nu \prec \la + \mu} \gr V_{\msp_{2n}}(\nu).$$
By Theorem \ref{ExtraThm} the projection of the vector $f^{\tilde\la}v_{\la} \otimes f^{\tilde\mu}v_{\mu}$ on $\gr V_{\msp_{2n}}(\la - \tilde\la) \otimes \gr V_{\msp_{2n}}(\mu - \tilde\mu)$ is non-zero and equal to a multiple of $v_{\la - \tilde\la} \otimes v_{\mu - \tilde\mu}$. By Theorem \ref{EvenSubmodule} under the action of $S(\mathfrak{n}^{-})$ the vector $v_{\lambda - \tilde\lambda} \otimes v_{\mu - \tilde\mu}$ generates a submodule isomorphic to $\gr V_{\msp_{2n}}(\lambda + \mu - \tilde\lambda - \tilde\mu) = \gr V_{\msp_{2n}}(\lambda + \mu - \tilde\nu)$. Since $f^{\tilde\nu}(v_{\lambda} \otimes v_{\mu})$ is a sum of linearly independent vectors
$$f^{\tilde\nu}(v_{\lambda} \otimes v_{\mu}) = \sum\limits_{\tilde\lambda + \tilde\mu = \tilde\nu} C_{\tilde\lambda, \tilde\mu} f^{\tilde\lambda}v_{\lambda} \otimes f^{\tilde\mu}v_{\mu},$$
for some constants $C_{\tilde\lambda, \tilde\mu}$, then $f^{\tilde\nu}(v_{\lambda} \otimes v_{\mu})$ also generates a submodule isomorphic to $\gr V_{\msp_{2n}}(\lambda + \mu - \tilde\nu)$. It follows that we can combine the results over all $\tilde\nu \in \delta(\la + \mu)$ and conclude that $v_{\lambda} \otimes v_{\mu}$ generates $S(\tilde\n^{-})$-module $\text{gr } \mathcal{V}_{0}(\lambda + \mu).$
\end{proof}

\section{Main Theorems}
In this section we will prove Theorems \ref{OddFFLV} and \ref{OddMinkowski}. For Theorem \ref{OddFFLV} we will follow the proof strategy of the Theorem \ref{EvenFFLV} from \cite{FFL2} with certain adjustments. Since the main part of the argument remains the same, we do not repeat all the details, referring to \cite{FFL2} where possible.

By Theorem \ref{ExtraThm} we have $S(\tilde\n^{-}) v_{\la} = \gr \mathcal{V}_{0}(\la)$. Therefore we can write $\gr \mathcal{V}_{0}(\la)$ as $S(\tilde\n^{-}) / I(\la)$ for some ideal $I(\la)$. By examining the weights we see that the vector $v_{\la} \in \mathcal{V}_{0}(\lambda)$ vanishes under the action of certain operators,
$$ f_{i, j}^{m_i + ... + m_j + 1} v_{\la} = 0, \;\; f_{i, \ol{i}}^{m_i + ... + m_n + 1}v_{\la} = 0.$$
It follows that they belong to $I(\la)$. Let
$$R = \text{span}\{ f_{i, j}^{m_i + ... + m_j + 1}, \; 1 \leq i \leq j \leq n - 1, \; f_{i, \ol{i}}^{m_i + ... + m_n + 1}, \; 1 \leq i \leq n \}.$$
As a byproduct, during the proof of Theorem \ref{OddFFLV}, we will obtain the following statement.
\begin{lem}
$I(\la) = S(\tilde\n^{-})(U(\tilde\n^{+}) \circ R).$
\end{lem}

We use differential operators $\pa_{\al}$ defined for $\msp_{2n}$ roots $\al \in R^{+}$ as
$$
\pa_\al f_\beta =
\begin{cases}
f_{\beta - \al},\  \text{ if }  f_{\beta - \al} \in \tilde\n^{-},\\
0,\ \text{ otherwise}.
\end{cases}
$$
We also introduce one new differential operator
$$\pa_{1, n} = \ad\; E_{2n + 1, n + 1},$$
with the key property $\pa_{1, n} f_{1, \ol{1}} = [E_{2n + 1, n + 1}, E_{n + 1, 1}] = E_{2n + 1, 1} = f_{1, n}$.
We also set up the notation,
$$s_{i, \bullet} = \sum_{j = i}^{n} s_{i, j} + \sum_{j = i}^{n} s_{i, \ol{j}}, \;\; s_{\bullet, i} = \sum_{j = 1}^{i} s_{j, i},$$
$$d(\bs) = (s_{n, \bullet}, s_{n - 1, \bullet}, ..., s_{1, \bullet}).$$
Write $d(\bs) > d(\bt)$ if the list $d(\bs)$ is greater than $d(\bt)$ according to the standard lexicographical order. Define the order on the variables $f_{\al} \in S(\tilde\n^{-})$.
\begin{equation*}
\begin{array}{rcl}
& f_{n, \ol{n}} > f_{n, n} > & \\
f_{n - 1,\overline{n - 1}} & > f_{n - 1, \ol{n}} > f_{n - 1, n} > & f_{n - 1, n - 1} > \\
f_{n - 2, \overline{n - 2}} > f_{n - 2, \overline{n - 1}} & > f_{n - 2, \ol{n}} > f_{n - 2, n} > & f_{n - 2, n - 1} > f_{n - 2, n - 2} > \\
\ldots & > \ldots > & \ldots >\\
f_{1, \overline{1}} > f_{1, \overline{2}} > \ldots & > f_{1, \ol{n}} > f_{1, n} > & \ldots > f_{1,2} > f_{1,1}.
\end{array}
\end{equation*}
The order $>$ extends to a monomial order on $S(\tilde\n^{-})$.

We will use one more important monomial order on $S(\tilde\n^{-})$. For $\bs, \bt \in \bz_{\geq 0}^{n(n + 1)}$ denote $\bs \succ \bt$ if one of the following conditions hold:
\begin{itemize}
\item[{\it a)}] $\deg(\bs) > \deg(\bt)$;
\item[or {\it b)}] total degree is the same, but $d(\bs) < d(\bt)$;
\item[or {\it c)}] total degree is the same, $d(\bs) = d(\bt)$, but $f^\bs > f^\bt$.
\end{itemize}
We are ready to prove Theorem \ref{OddFFLV}.
\begin{proof}
From the papers \cite{Proctor} and \cite{ABS} we know that 
$$\dim \mathcal{V}_{0}(\la) = |FFLV_{2n + 1}^{\bz}(\la)|.$$
Therefore, it is enough to prove that the set of vectors
$$f^{\bs} v_{\la}, \; \bs \in FFLV_{2n + 1}^{\bz}(\la),$$
spans $\gr \mathcal{V}_{0}(\la)$. During the proof of Theorem \ref{ExtraThm} we have obtained 
$$S(\tilde\n^{-}) v_{\la} = \gr \mathcal{V}_{0}(\la).$$
As shown in \cite{FFL2}, it suffices to prove that any element $f^{\bs} \in S(\tilde\n^{-})$ with $\bs \in \bz_{\geq 0}^{n(n + 1)} \setminus FFLV_{2n + 1}^{\bz}(\la)$ has a \textit{straightening law}:
$$f^{\bs} + \sum_{\bs \succ \bt} f^{\bt} \in I(\la),$$
where $\succ$ is a monomial order.

Before presenting a straightening law we consider a sequence of reductions. Since $\bs \in \bz_{\geq 0}^{n(n + 1)} \setminus FFLV_{2n + 1}^{\bz}(\la)$, then for some symplectic Dyck path $p = (p(0), ..., p(k))$ the inequality is violated. Since the order $\succ$ is monomial, it is enough to find straightening law for $\bs$ restricted to the support of path $p$. Suppose that path $p$ starts at $p(0) = \al_{i, i}$. If $i \neq 1$, then we can apply induction on $n$ by using the inclusion $\msp_{2(n - 1) + 1} \subset \msp_{2n + 1}$. From now on we assume that $p(0) = \al_{1, 1}$. Suppose that path $p$ ends at $p(k) = \al_{j, j}$. Note that any path ending at $\al_{n, n}$ can be extended to end at $\al_{n, \ol{n}}$ without weakening the inequality. Thus we can assume that $j \neq n$. According to \cite{FFL2}, it follows that the straightening law of $\msl_{n}$ can be applied, since the support of $p$ contains only roots of the subalgebra $\msl_{n}$. Therefore the case $p(k) = \al_{j, j}$ is solved. From now on we assume that $p(k) = \al_{j, \ol{j}}$.

To sum up the reductions of the previous paragraph, it is enough to consider lattice points $\bs$ with the support on a single path $p$ that starts at $p(0) = \al_{1, 1}$ and ends at $p(k) = \al_{j, \ol{j}}$. Moreover, the inequality 
$$\sum_{\al \in p} s_{\al} \leq m_1 + ... + m_n$$
is violated. Therefore $\Sigma = \sum_{\al \in p} s_{\al} \geq m_1 + ... + m_n + 1$. And it follows that $f_{1, \ol{1}}^{\Sigma} \in I(\la)$. Define the differential operators
$$
\Delta_1 := \pa_{1, i - 1}^{s_{\bullet,\bar i} + s_{i,\bullet}}
\underbrace{
\pa_{i + 1, \ol{i + 1}}^{s_{\bullet, i}} \ldots \pa_{n, \ol{n}}^{s_{\bullet, n - 1}}
}_{\delta_3} \cdot
$$ 
$$
\cdot\underbrace{
\underline{\pa_{1, n}^{s_{\bullet, n}}} \pa_{1,n-1}^{s_{\bullet, n - 1} + s_{\bullet, \ol{n}}}
\ldots
\pa_{1,i}^{s_{\bullet,i}+s_{\bullet,\ol{i+1}}}
}_{\delta_2}
\underbrace{
\pa_{1,\bar i}^{s_{\bullet,i-1}}\dots \pa_{1,\bar 3}^{s_{\bullet,2}}\pa_{1,\bar 2}^{s_{\bullet,1}}
}_{\delta_1}
$$
(so $\Delta_1 := \pa_{1, i - 1}^{s_{\bullet, \ol{i}} + s_{i, \bullet}} \delta_3 \delta_2 \delta_1$)
and
$$
\Delta_2:=\pa_{1,1}^{s_{2,\bullet}}\pa_{1,2}^{s_{3,\bullet}}\dots \pa_{1,i-2}^{s_{i-1,\bullet}}.
$$
Here we underlined the only difference from operators defined in \cite{FFL2}, the term $\pa_{1, n}^{s_{\bullet, n}}$. We claim that the element $\Delta_2 \Delta_1 f^{\Sigma}_{1, \ol{1}}$ provides a straightening law
$$\Delta_2 \Delta_1 f^{\Sigma}_{1, \ol{1}} = c_{\bs} f^{\bs} + \sum_{\bs \succ \bt} c_{\bt} f^{\bt} \in I(\la).$$
The following argument is identical to the one presented in \cite{FFL2}. We can show by induction that
$$\delta_1 (f_{1, \ol{1}}) = f^{s_{\bullet, 1}}_{1, 1} f^{s_{\bullet, 2}}_{1, 2} ... f^{s_{\bullet, i - 1}}_{1, i - 1} f^{\Sigma - s_{\bullet, 1} - ... - s_{\bullet, i - 1}}_{1, \ol{1}}.$$
By noticing that all operators from $\pa_{1, i - 1}^{s_{\bullet, \ol{i}} + s_{i, \bullet}} \delta_3\delta_2$ kill the terms $f_{1, j},\; j \leq i - 1$, we write
$$\Delta_1 (f_{1, \ol{1}}) = f^{s_{\bullet, 1}}_{1, 1} f^{s_{\bullet, 2}}_{1, 2} ... f^{s_{\bullet, i - 1}}_{1, i - 1} \pa_{1, i - 1}^{s_{\bullet, \ol{i}} + s_{i, \bullet}} \delta_3\delta_2 (f^{\Sigma - s_{\bullet, 1} - ... - s_{\bullet, i - 1}}_{1, \ol{1}}).$$
Next we obtain
$$\delta_2 (f^{\Sigma - s_{\bullet, 1} - ... - s_{\bullet, i - 1}}_{1, \ol{1}}) = f^{s_{\bullet, \ol{i}} + s_{\bullet, i - 1}}_{1, \ol{i + 1}} ... f^{s_{\bullet, \ol{n}} + s_{\bullet, n - 1}}_{1, \ol{n}} f^{s_{\bullet, n}}_{1, n} f^{s_{\bullet, \ol{i}}}_{1, \ol{1}} + $$
$$ + \sum_{\bt} (\text{terms with non-zero entries in } \bt \text{ below row } i).$$
Note that no operator from $\Delta_2, \delta_3, \pa_{1, i - 1}$ can decrease the sum of entries of $\bt$ in rows below $i$. Since the monomial order $\succ$ is smaller for elements $f^{\bt}$ with larger $d(\bt)$, we conclude that terms $f^{\bt}$ with non-zero entries below row $i$ will remain smaller than $f^{\bs}$ after the action of all operators.

The operator $\delta_3$ kills the term $f_{1, n}$. Therefore the result of its action on $\delta_2 (f^{\Sigma - s_{\bullet, 1} - ... - s_{\bullet, i - 1}}_{1, \ol{1}})$ is analogous to the $\msp_{2n}$ case:
$$\delta_3 \delta_2 (f^{\Sigma - s_{\bullet, 1} - ... - s_{\bullet, i - 1}}_{1, \ol{1}}) = f^{s_{\bullet, i}}_{1, i} f^{s_{\bullet, i + 1}}_{1, i + 1} ... f^{s_{\bullet, n}}_{1, n}   f^{s_{\bullet, \ol{n}}}_{1, \ol{n}} ... f^{s_{\bullet, \ol{i + 1}}}_{1, \ol{i + 1}} f^{s_{\bullet, \ol{i}}}_{1, \ol{1}} + $$
$$ + \sum_{\bt} (\text{terms with non-zero entries in } \bt \text{ below row } i).$$
The operator $\pa_{1, i - 1}$ either moves the element from row 1 to row $i$ or it changes $f_{1, \ol{1}}$ to $f_{1, \ol{i}}$. To maximize the monomial order $\succ$ we have to minimize the sum of entries in row $i$ (and keep entries in rows below $i$ equal to $0$), which is achieved when $\pa_{1, i - 1}$ is applied $s_{\bullet, \ol{i}}$ times to $f^{s_{\bullet, \ol{i}}}_{1, \ol{1}}$. Denote 
$$f^{\bs'} = f^{s_{\bullet, 1}}_{1, 1} f^{s_{\bullet, 2}}_{1, 2} ... f^{s_{\bullet, q} - s_{i, q}}_{1, q} f^{s_{i, q}}_{i, q} ... f^{s_{i, \ol{i}}}_{i, \ol{i}},$$
where $q = \min \{ j \vert \; \al_{i, j} \in p\}$ from the poset $1 < 2 < ... < n < \ol{n} < ... < \ol{1}$. Then we can show that
$$\Delta_1(f_{1, \ol{1}}^{\Sigma}) = f^{\bs'} + $$
$$ + \sum_{\bs' \succ \bt}(\text{terms with }  d(\bt) > d(\bs') \text{ or with } f_{i, i}^{t_{i, i}}f_{i, i + 1}^{t_{i, i + 1}} ... f_{i, \ol{i}}^{t_{i, \ol{i}}} < f_{i, i}^{s_{i, i}} ... f_{i, \ol{i}}^{s_{i, \ol{i}}}).$$
Finally, when applied to $f^{\bs'}$, the operator $\Delta_2$ moves elements from top row 1 to rows between 1 and $i$. To maximize the monomial order $\succ$ it will leave non-zero entries on the support of the considered path as shown in \cite{FFL2}. 
\end{proof}
\begin{rem}
It is not hard to show that simultaneous swap of the roots $\al_{i, n}$ and $\al_{i, \ol{n}}$ for all $i$ (and of the elements $f_{i, n}$ and $f_{i, \ol{n}}$ for all $i$) in the root poset definition also gives a valid FFLV polytope and PBW bases. That means, we can instead define
$$f_{i, n} = f_{\al_i + ... + \al_n}, \;\; f_{i, \ol{n}} = E_{2n + 1, i},$$
but we have to modify the operators $\Delta_1, \Delta_2$ accordingly.
\end{rem}
\begin{rem}
Note that specializing the argument to the case $s_{i, n} = 0$ for all $i$ gives a basis of subrepresentation $V_{\msp_{2n}}(\la)$  and thus proves Theorem \ref{EvenFFLV}.
\end{rem}
\begin{cor}
$\ch_{q} \mathcal{V}_{0}(\la) = \sum\limits_{\bs \in FFLV^{\bz}_{2n + 1}(\la)} e^{\la - \wt(\bs)} q^{\deg(\bs)}.$
\end{cor}

Finally, we show that Theorem \ref{OddMinkowski} is a corollary of Theorem \ref{EvenMinkowski}.
\begin{proof}
Let $\la = (\la_1, ..., \la_n, 0)$ be the weight of $\msp_{2n + 2}$. Consider the polytope $P_{2n + 1}(\la)$ defined as the the intersection of the polytope $FFLV_{2n + 2}(\la)$ with the subspace $\{ s_{i, \ol{n + 1}} = 0, \; i = 1, 2, ..., n + 1\}$,
$$P_{2n + 1}(\la) = FFLV_{2n + 2}(\la) \cap \{ s_{i, \ol{n + 1}} = 0 \; \forall i\}.$$
Convex polytopes $P_{2n + 1}(\la)$ and $FFLV_{2n + 1}(\la)$ are equal as polytopes in $\R^{n(n + 1)}$. By Theorem \ref{EvenMinkowski} and the fact that $s_{i, \ol{n + 1}} \geq 0 \; \forall i$ we get
$$P_{2n + 1}^{\bz}(\la) + P_{2n + 1}^{\bz}(\mu) = P_{2n + 1}^{\bz}(\la + \mu),$$
which translates to the equality
$$FFLV_{2n + 1}^{\bz}(\la) + FFLV_{2n + 1}^{\bz}(\mu) = FFLV_{2n + 1}^{\bz}(\la + \mu).$$
\end{proof}

\printbibliography

@article{Proctor,
author = {Proctor, Robert A.},
journal = {Inventiones mathematicae},
number = {2},
pages = {307-332},
title = {Odd symplectic groups.},
volume = {92},
year = {1988}
}

@article{FFL1,
    author = "Evgeny Feigin and Ghislain Fourier and Peter Littelmann",
    title = "PBW filtration and bases for irreducible modules in type $A_n$",
    journal = "Transformation Groups",
    volume = "16",
    pages = "71--89",
    year = "2011"
}

@article{FFL2,
    author = {Feigin, Evgeny and Fourier, Ghislain and Littelmann, Peter},
    title = "{PBW Filtration and Bases for Symplectic Lie Algebras}",
    journal = {International Mathematics Research Notices},
    volume = {2011},
    number = {24},
    pages = {5760-5784},
    year = {2011},
    month = {02}
}

@article{IM,
     author = {Makhlin, Igor},
     title = {FFLV-type monomial bases for type $B$},
     journal = {Algebraic Combinatorics},
     pages = {305--322},
     volume = {2},
     number = {2},
     year = {2019}
}

@article{ABS,
    title = {Gelfand–Tsetlin polytopes and Feigin–Fourier–Littelmann–Vinberg polytopes as marked poset polytopes},
    journal = {Journal of Combinatorial Theory, Series A},
    volume = {118},
    number = {8},
    pages = {2454--2462},
    year = {2011},
    author = {Federico Ardila and Thomas Bliem and Dido Salazar}
}

@article{MinkowskiSums,
    author = {Fourier, Ghislain},
    title = {Marked poset polytopes: Minkowski sums, indecomposables, and unimodular equivalence},
    journal = {Journal of pure and applied algebra},
    volume = {220},
    number = {2},
    pages = {606--620},
    year = {2016}
}

@article{Fang,
   title={The Minkowski Property and Reflexivity of Marked Poset Polytopes},
   volume={27},
   number={1},
   journal={The Electronic Journal of Combinatorics},
   author={Fang, Xin and Fourier, Ghislain and Pegel, Christoph},
   year={2020},
   month={Jan}
}

@misc{Kambaso,
      title={Homogeneous Bases for Demazure Modules}, 
      author={Kunda Kambaso},
      year={2021},
      eprint={2007.11054},
      archivePrefix={arXiv},
      primaryClass={math.RT}
}

@misc{PosetDynamics,
      title={Order polynomial product formulas and poset dynamics}, 
      author={Sam Hopkins},
      year={2020},
      eprint={2006.01568},
      archivePrefix={arXiv},
      primaryClass={math.CO}
}

@book{Weyl,
  year = {2016},
  month = jun,
  publisher = {Princeton University Press},
  author = {Hermann Weyl},
  title = {The Classical Groups}
}

@misc{Vinberg,
	title={On some canonical bases of representation spaces of simple Lie algebras},
	author={E. Vinberg},
	year={2005},
	howpublished={conference talk, Bielefeld}
}
\end{document}